\documentclass[a4paper,10pt]{amsart}
\usepackage[utf8x]{inputenc}
\usepackage[T1]{fontenc}
\usepackage{amsmath}
\usepackage{amsfonts}
\usepackage{amssymb}

\newtheorem{thrm}{Theorem}[section]

   \newtheorem{fact}[thrm]{Proposition}
   
   \newtheorem{col}[thrm]{Corollary}

\theoremstyle{definition}

\newtheorem{ex}[thrm]{Example}

\newcommand{\cc}{\mathfrak{c}}
\newcommand{\PP}{\mathcal{P}}
\newcommand{\RR}{\mathbb{R}}

\begin{document}

\title[Some remarks on universality properties of $\ell_\infty / c_0$]{Some remarks on universality properties of $\ell_\infty / c_0$}

\author[M.\ Krupski]{Miko\l aj Krupski}
\address{Institute of Mathematics\\ Polish Academy of Sciences\\ \newline Ul. \'Sniadeckich 8\\00--956 Warszawa\\ Poland }
\email{krupski@impan.pl}

\author[W.\ Marciszewski]{Witold Marciszewski}
\address{Institute of Mathematics\\
University of Warsaw\\ Banacha 2\newline 02--097 Warszawa\\
Poland}
\email{wmarcisz@mimuw.edu.pl}

\date{}

\begin{abstract}
We prove that if $\cc$ is not a Kunen cardinal, then there is a uniform Eberlein compact space $K$ such that
the Banach space $C(K)$ does not embed isometrically into $\ell_\infty/c_0$.
We prove a similar result for isomorphic embeddings.
We also construct a consistent example of a uniform Eberlein
compactum whose space of continuous functions embeds isomorphically into $\ell_\infty/c_0$, but fails to embed isometrically.
As far as we know it is the first example of this kind.
\end{abstract}

\subjclass[2010]{Primary 46B26, 46E15; Secondary 03E75}

\keywords{$C(K)$ spaces, uniform Eberlein compact, Kunen cardinal,
universal space}

\maketitle

\section{Introduction}
For a compact space $K$ we denote by $C(K)$ the Banach space of continuous functions on $K$ with the supremum norm.
We say that a Banach space $X$ is universal (isometrically universal) for the class $\mathcal{K}$ of compact spaces if
for any $K\in\mathcal{K}$ the space $C(K)$ embeds isomorphically (isometrically) into the space $X$.

In this note we deal with the universality properties of the space $\ell_\infty/c_0$.
Let us recall a classical result of Parovi\v{c}enko that the space $C(\beta\omega\setminus \omega)$, which is isometric to $\ell_\infty/c_0$,
is isometrically universal for the class of compact spaces
of weight $\aleph_1$. It is natural to ask are there any other classes of spaces for which $\ell_\infty/c_0$ is universal.
Clearly this question makes sense only if we restrict ourselves to spaces of weight $\leqslant\cc$.

It has been shown by C.\ Brech and P.\ Koszmider in \cite{BK} that consistently it is not the case for the class of uniform Eberlein compacta.
In fact they proved more: consistently there is no universal space for the class of uniform Eberlein compacta.

Recently S.\ Todor\v{c}evi\'c found a beautiful connection between universality properties of $\ell_\infty / c_0$ and properties
of $\sigma$-field of subsets of $\RR^n$ generated by sets of the form $A_1\times \ldots \times A_n$ where $A_1,\ldots ,A_n\subseteq \RR$ (see \cite{T}).
More precisely, he proved that if $\cc$ is not a Kunen cardinal\footnote{see section 2 for definitions and notation},
then $\ell_\infty / c_0$ is not an isometrically universal space
for Corson compacta. He also proved, under another set-theoretical assumption involving
$\sigma$-fields of subsets of $\RR^n$, a similar result for isomorphic rather than isometric embeddings.

In this note we strengthen Todor\v{c}evi\'c's results by a simple modification of his argument. We prove that if $\cc$ is not a Kunen
cardinal, then $\ell_\infty / c_0$ is not a universal space for the class of uniform Eberlein compacta. In particular, by a different approach,
we get a mentioned result of C.\ Brech and P.\ Koszmider.
Our considerations lead to a consistent example of a space of continuous functions on a uniform Eberlein compactum, which distinguishes
isometric from isomorphic embeddings into $\ell_\infty / c_0$. In particular the uniform Eberlein compactum we get is not a continuous
image of $\beta\omega\setminus\omega$ yet its space of continuous functions
embeds isomorphically into $\ell_\infty/c_0$. As far as we know it is the first example of this kind.

\section{Preliminaries}
We use a standard set-theoretical and topological notation. Given a set $A$ and a positive integer $n$ we denote by $[A]^n$ ($[A]^{\leqslant n}$)
the family
of all subsets of $A$ of cardinality $n$ ($\leqslant n$). By $\cc$ we denote the cardinal number $2^{\aleph_0}$ (the continuum).

Let us recall that a compact space is \textit{uniform Eberlein} if it is homeomorphic to a subset of a Hilbert space in its weak topology
(see \cite{Ne}). Equivalently, a space is a uniform Eberlein compactum if it can be embedded into a space
$$B(\Gamma)=\{x\in[-1,1]^\Gamma:\sum_{\gamma\in\Gamma}|x_\gamma|\leqslant 1\}$$
for some index set $\Gamma$. Indeed, the above space is homeomorphic to a ball in the space $(\ell_2, weak)$.
A well known example of a uniform Eberlein compactum is the following. Take
a natural number $n$ and an infinite set $\Gamma$ and put
$$\sigma_n(\Gamma)=\{x\in\{0,1\}^\Gamma:\lvert\{\gamma\in\Gamma:x_\gamma\neq 0\}\rvert\leqslant n\}\text{.}$$
This space, being homeomorphic to $B(\Gamma)\cap\{0,\frac{1}{n}\}^{\Gamma}$, is uniform Eberlein compact.
The following, probably well known fact about uniform Eberlein compacta will be useful. It says that
this class of spaces is closed under taking a one point compactification of a discrete sum.

\begin{fact}\label{sumakompaktow}
Let $T$ be an arbitrary set of indexes and suppose $K_t$ is a uniform Eberlein compactum for every $t\in T$.
Then the space
$K=\bigoplus_{t\in T}K_t\cup\{\infty\}$ (i.e. one point compactification of a discrete sum
of $K_t\text{'s}$) is uniform Eberlein compact.
\end{fact}
\begin{proof}
For every $t\in T$, there is an embedding $h_t:K_t\rightarrow B(\Gamma_t)$, for some $\Gamma_t$. Let
$\Gamma$ be a disjoint union of $\{\Gamma_t:t\in T\}$ and $T$. Then $h:K\rightarrow B(\Gamma)$ defined by
$$h(x)_\gamma=\left\{
    \begin{array}{llll}
        \frac{1}{2} h_t(x)_\gamma  & \mbox{if } \gamma\in \Gamma_t \\
        0 & \mbox{if } \gamma\in \Gamma_s,\;s\neq t \\
        \frac{1}{2} & \mbox{if } \gamma=t \\
        0 & \mbox{if } \gamma \in T,\;\gamma\neq t \\
    \end{array}
\right.$$
for $x\in K_t$ and $h_\gamma(\infty)=0$ for every $\gamma\in \Gamma$ is the desired embedding.
\end{proof}

For an arbitrary set $\Gamma$ and a natural number $k\geqslant 2$, we denote by $\mathcal{P}^k(\Gamma)$ the
$\sigma$-field generated by sets of the form
$A_1\times\ldots\times A_k$ where $A_1,\ldots , A_k\subseteq\Gamma$.
Following \cite{ARP} we call a cardinal $\kappa$ \textit{Kunen} if $\PP(\kappa\times\kappa)=\PP^2(\kappa)$.
It is clear that for an arbitrary set $\Gamma$, the equality $\PP(\Gamma\times\Gamma)=\PP^2(\Gamma)$ depends only on
cardinality of $\Gamma$, so it holds if and only if $\lvert\Gamma \rvert$ is Kunen. We refer the reader to \cite{ARP} for
more on Kunen cardinals. Let us only mention here that the statement '$\cc$ is a Kunen cardinal' is independent of ZFC.

Finally, let us recall that elements of $\ell_\infty/c_0$ are of the form $[x]=\{y\in \ell_\infty:(x-y)\in c_0\}$, where $x\in \ell_\infty$.
For $[x]\in \ell_\infty/c_0$ we have $\lVert [x]\rVert=\limsup_n \lvert x(n)\rvert$.

\section{Proofs}

In this main section of our note we prove a strengthening of two theorems of Todor\v{c}ević from \cite{T}.
Theorems \ref{isometr} and \ref{isomorph} below are counterparts of Theorems 4.1 and 4.3 in \cite{T}, respectively.
Although the main idea in our proofs is the same as in \cite{T}, for the sake of completeness,
we decided to enclose here quite detailed reasonings.
For a binary relation $E\subseteq\RR^2$ and a natural number $n$ we denote by $K_n(E)$ the following space
$$K_n(E)=\{\chi_A\in \{0,1\}^{\RR}:A\in[\RR]^{\leqslant n},\; \forall a,b\in A\;\; a<b\Rightarrow (a,b)\in E\}.$$
It is a compact subspace of $\{0,1\}^\RR$ and moreover it is a uniform Eberlein compact space since it embeds into the space $\sigma_n(\RR)$.

As we will see, the following easy proposition plays a key role in the proof of the next theorem (see \cite{T}).
\begin{fact}\label{1}
Let $(X,\tau)$ be a separable, metrizable topological space.
Let $f:\RR\rightarrow X$ be an injection and $S\subseteq \RR^2$ be such that $f\times f[S]$ is a Borel subset
of $(f\times f[\RR^2],\tau\times\tau)$. Then $S\in\PP^2(\RR)$.
\end{fact}

\begin{thrm}\label{isometr}
Suppose that, for every uniform Eberlein compact space $K$ of weight at most $\cc$, the space $C(K)$ embeds isometrically into $\ell_\infty / c_0$.
Then $\mathfrak{c}$ is a Kunen cardinal.
\end{thrm}
\begin{proof}
Suppose the contrary and let $E$ witnesses that $\cc$ is not Kunen. Since $E\cap \{(a,b):a=b\}$ is in $\PP^2(\RR)$, one of the sets
$E_0=E\cap\{(a,b):a<b\}$ or $E_1= E\cap\{(a,b):a>b)\}$ is not in $\PP^2(\RR)$. By symmetry we can assume it is $E_0$.
Consider the space $K_2(E_0)$ which is uniform Eberlein compact. Now the proof goes as in \cite{T}.
We define an injection $\phi:\RR\rightarrow C(K_2(E_0))$ by $\phi(r)=f_r$, where $f_r(x)=x(r)$ for $x\in K_2(E_0)$. Let
$T:C(K_2)\rightarrow \ell_\infty / c_0$ be an isometry (which exists by our assumption) and let $\psi:\ell_\infty / c_0\rightarrow \ell_\infty$
be an arbitrary injection (a selector).
For $a<b$ we have
$$(a,b)\in E_0\text{  iff  } \chi_{\{a,b\}}\in K_2(E_0) \text{ iff } \lVert T(f_a)+T(f_b)\rVert=\lVert f_a+f_b \lVert>1\text{  iff  } $$
$$\limsup_n\lvert(\psi\circ T(f_a))(n)+(\psi\circ T(f_b))(n)\rvert>1$$
so putting $g=\psi\circ\ T\circ\phi$ we have that
$$g\times g[\RR^2]\cap \{(x,y)\in (\RR^\omega)^2:\limsup_n \lvert x(n)+y(n)\rvert>1\}$$
is a Borel subset of $g\times g[\RR^2]\subseteq\ell_\infty\times \ell_\infty$
in the topology inherited from $(\RR^\omega)^2$. Hence, by Proposition \ref{1},
$$(g\times g)^{-1}[\{(x,y)\in (\RR^\omega)^2:\limsup_n \lvert x(n)+y(n)\rvert>1\}]\in \PP^2(\RR).$$
Since
$$E_0=(g\times g)^{-1}[\{(x,y)\in (\RR^\omega)^2:\limsup_n \lvert x(n)+y(n)\rvert>1\}]\cap \{(x,y)\in \RR^2:x<y\}$$
and since $\{(x,y)\in \RR^2:x<y\}\in \PP^2(\RR)$ we conclude that $E_0\in\PP^2(\RR)$, a contradiction.

\end{proof}

Notice that the result follows also under a weaker assumption: for every \textit{scattered} uniform Eberlein compactum (of height $3$)
its space of continuous functions embeds isometrically into $\ell_\infty / c_0$.

To prove the result about isomorphic embeddings we need the following modification of Proposition \ref{1}.
Here $f^n$ denotes a function which is the product
$f\times \ldots \times f$ of $n$ many copies of $f$.
\begin{fact}\label{2}
 Let $(X,\tau)$ be a separable, metrizable topological space.
Let $f:\RR\rightarrow X$ be an injection and $A,B\subseteq \RR^n$ for some integer $n\geqslant 2$.
If $f^n[A]$ is separated from $f^n[B]$ by a Borel set in $X^n$, then $A$ can be separated from $B$ by a member of
$\PP^n(\RR)$.
\end{fact}

Following \cite{T}, for any binary relation $E\subseteq\RR^2$ and integer $n\geqslant 2$, we denote by $E^{[n]}$ the following set
$$E^{[n]}=\{(x_1,\ldots x_n)\in\RR^n:\forall i<j\; x_i\neq x_j \text{ and } (x_i,x_j)\in E\}.$$
In particular, $<^{[n]}$ is the set of all strictly increasing sequences of reals of length $n$.
The complementary relation to $E$ we will denote by $E^c$.

\begin{thrm}\label{isomorph}
Suppose that, for every uniform Eberlein compact space $K$ of weight at most $\cc$, the space $C(K)$ embeds isomorphically into $\ell_\infty / c_0$.
Then for every binary relation $E\subseteq \RR^2$ and for all but finitely many positive integers n,
the set $E^{[n]}$ can be separated from $(E^c)^{[n]}$ by a member of
$\PP^n(\RR)$
\end{thrm}

\begin{proof}
Let $E\subseteq\RR^2$ be arbitrary. First we will prove the theorem for the relation
$$E_0=(E\cap <)=\{(x,y)\in \RR^2:x<y\text{ and }(x,y)\in E\}.$$
Since $E_0^{[n]}\subseteq <^{[n]}$ and $<^{[n]}\in \PP^n(\RR)$, to get a desired separation between $E_0^{[n]}$ and $(E_0^c)^{[n]}$,
it is enough to separate $E_0^{[n]}$ from $(E_0^c)^{[n]}\cap <^{[n]}$ by a member of $\PP^n(\RR)$. To this end let us consider
$$K=\bigoplus_{n=1}^{\infty}K_n(E_0)\cup\{\infty\}$$
which is a one point compactification of a discrete sum of spaces $K_n(E_0)$. By Proposition \ref{sumakompaktow},
$K$ is a uniform Eberlein compactum.

Let
$T:C(K)\rightarrow l_{\infty}/c_0$ be an isomorphism (which exists by the assumption). Then there exists a positive integer $k$
such that $k>\lVert T\rVert\lVert T^{-1}\rVert$.
Let $n\geqslant k$ be arbitrary.

As in the previous proof we define an injection $\phi:\RR\rightarrow C(K)$ by $\phi(r)=f_r$ where $f_r:K\rightarrow \RR$ is defined by
$$f_r(x)=\left\{
  \begin{array}{ll}
    x(r) & \mbox{if } x\in K_n(E_0) \\
    0 & \mbox{if } x\in K\setminus K_n(E_0)
\end{array}
\right.$$

Now, if
$(x_1,\ldots ,x_n) \in E_0^{[n]}$ then $\chi_{\{x_1,\ldots ,x_n\}}\in K_n(E_0)\subseteq K$ so $\lVert f_{x_1}+\ldots +f_{x_n}\rVert=n$.
Hence
$$n=\lVert f_{x_1}+\ldots +f_{x_n}\rVert=\lVert T^{-1}T(f_{x_1}+\ldots +f_{x_n})\rVert\leqslant \lVert T^{-1}\rVert
\lVert T(f_{x_1}+\ldots +f_{x_n})\rVert \text{ ,}$$
so
$$\lVert T(f_{x_1}+\ldots +f_{x_n})\rVert\geqslant\frac{n}{\lVert T^{-1}\rVert}\geqslant\frac{k}{\lVert T^{-1}\rVert}>\lVert T\rVert.$$
If $(x_1,\ldots ,x_n) \in (E_0^c)^{[n]}\cap <^{[n]}$ then, for any $x\in K$, at most one function $f_{x_i}$ has value $1$ at $x$, so
$\lVert f_{x_1}+\ldots +f_{x_n}\rVert=1$.
Hence
$$\lVert T(f_{x_1}+\ldots +f_{x_n})\rVert\leqslant \lVert T\rVert\lVert f_{x_1}+\ldots +f_{x_n}\rVert=\lVert T\rVert.$$
For an arbitrary injection $\psi:\ell_\infty / c_0\rightarrow \ell_\infty$
we put $g=\psi\circ T\circ\phi$ and conclude that
the set
$$\{(x_1,\ldots , x_n)\in (\ell_\infty)^n:\limsup_m\lvert x_1(m)+\ldots +x_n(m)\rvert >\lVert T\rVert\}$$
separates $g^n[E_0^{[n]}]$ from $g^n[(E_0^c)^{[n]}\cap <^{[n]}]$ (recall that by $g^n$ we mean the product of $n$ many copies of $g$).
Since it is Borel in the topology inherited from $(\RR^{\omega})^n$ the result follows from Proposition \ref{2}.

By symmetry the above argument works also for the relation
$$E_1=(E\cap >)=\{(x,y)\in \RR^2:x>y\text{ and }(x,y)\in E\}.$$
We only need to change $<$ for $>$ in the definition of $K_n(E)$.

To get the result in the full generality, observe that if $n$ is sufficiently large then any sequence $(x_1,\ldots,x_n)\in \RR^n$ has a subsequence
of length $k$ which is strictly increasing or strictly decreasing. That is, if for any $A\in [n]^k$ (we identify here $n$ with $\{1,\ldots ,n\}$)
we put
$$F_0(A)=\{(x_1,\ldots , x_n)\in \RR^n: \forall i,j\in A \;\; i<j \Rightarrow x_i<x_j\} \text{ and}$$
$$F_1(A)=\{(x_1,\ldots , x_n)\in \RR^n: \forall i,j\in A \;\; i<j \Rightarrow x_i>x_j\},$$
then
$$\RR^n=\bigcup_{A\in [n]^k}(F_0(A)\cup F_1(A)).$$
For $A\in [n]^k$ and $i\in\{0,1\}$ let $G_i(A)=E^{[n]}\cap F_i(A)$ and let $\pi_A:\RR^n\rightarrow \RR^A$
be the projection.

It is clear that we can identify $\pi_A(G_i(A))$ with $E_i^{[k]}$, so from the first part of the proof follows the existence of a set
$S_i(A)\in \PP^k(\RR)$ which separates $\pi_A(G_i(A))$ from $\pi_A((E_i^c)^{[n]})$.
It is not difficult to check that the set
$$\bigcup_{A\in [n]^k, i\in\{0,1\}}\pi_A^{-1}(S_i(A))$$
which is clearly in $\PP^n(\RR)$ separates $E^{[n]}$ from $(E^c)^{[n]}$.

\end{proof}

We should mention here, that consistently the thesis of \ref{isomorph} may not hold i.e. consistently there is a set
$E\subseteq\RR^2$ such that, for no natural number $n$, the set $E^{[n]}$ can not be separated from $(E^c)^{[n]}$
by a member of $\PP^n(\RR)$ (see \cite{T}, Remark 3.5). Thus, from what we proved it follows that consistently there exists
a uniform Eberlein compactum $K$ such that $C(K)$ is not isomorphically embeddable into $\ell_\infty/c_0$.

We shall show now that consistently there is a space of continuous functions on a uniform Eberlein compactum which
distinguishes isometric from isomorphic embeddings into $\ell_\infty/c_0$.

\begin{thrm}
If $\cc$ is not a Kunen cardinal then there exists a uniform Eberlein compact space $K$ such that the space $C(K)$
embeds isomorphically and does not embed isometrically into $\ell_\infty / c_0$.
\end{thrm}
\begin{proof}
 Since $\cc$ is not Kunen, there exists $E\subseteq \RR^2$ which is not in the $\sigma$-field $\PP^2(\RR)$.
 Without loss of generality, we can assume that $E\subseteq
 \{(a,b)\in\RR^2: a<b\}$, see the proof of Theorem \ref{isometr}.
 From this proof
it follows that $K=K_2(E)$ is a uniform Eberlein compact space
such that $C(K)$ does not embed isometrically into $\ell_\infty /
c_0$. On the other hand, $C(K)=C(K_2(E))$ always embeds
isomorphically into $\ell_\infty / c_0$, since it is isomorphic to
$c_0(\cc)$ (see \cite{M}, Theorem 1.1).
\end{proof}

\section{Remarks}
In this section we discuss shortly universality properties of $\ell_\infty/c_0$ for other two classes of compacta:
spaces which are continuous images of $\sigma_1(\cc)^\omega$ and so called AD-compacta.

It was shown by Y.\ Benyamini, M.E.\ Rudin and M.\ Wage in
\cite{BRW} that a space $K$ is uniform Eberlein compact of weight
$\leqslant\kappa$ if and only if it is a continuous image of a closed
subset of $\sigma_1(\kappa)^\omega$ (which is homeomorphic to
$A(\kappa)^\omega$, where $A(\kappa)$ denotes the one point
compactification of a discrete space of size $\kappa$). In the
same paper the authors asked whether we can replace a closed
subset of $\sigma_1(\kappa)^\omega$ by $\sigma_1(\kappa)^\omega$
itself. This question was answered in the negative by M.\ Bell in
\cite{Bell'96}. He considered a space homeomorphic to the space
$$ \{\chi_A\in \{0,1\}^{\omega_1}: A\in [\omega_1]^{\leqslant 2} \;\forall a,b \in A\; a<b \text{ iff } a\prec b\},$$
where $\prec$ is a well ordering on $\omega_1$ (\cite{Bell'96}, Example 3.1)
and proved that it is not a continuous image of $\sigma_1(\omega_1)^\omega$ (see Example \ref{example} below for a different space of this kind).
It is not difficult to see
that the above space is a continuous image of
the space $K_2(\prec)$, where $\prec$ is a well ordering on $\RR$. It was pointed out by S.\ Todor\v{c}evi\'c that consistently (in a model
obtained by adding more than continuum many reals) the set
$$E=\{(a,b)\in [0,1]^2: a<b \text{ iff } a\prec b\},$$
where $\prec$ is a well ordering on the interval $[0,1]$ is not in
the $\sigma$-field $\PP^2(\RR)$ (see \cite{T}, Remark 3.5). Thus
from Theorem \ref{isometr} it follows that (in this model) the
space $C(K_2(\prec))$ does not embed isometrically into
$\ell_\infty/ c_0$.

It turns out that $\ell_\infty/c_0$ is isometrically universal for the class of continuous images of $\sigma_1(\cc)^\omega$.
So we can distinguish the class of uniform Eberlein compacta of weight at most $\cc$ from the class of continuous images of
$\sigma_1(\cc)^\omega$ in terms of
universality properties of $\ell_\infty/c_0$.
\begin{fact}
If $K$ is a continuous image of $\sigma_1(\cc)^\omega$, then the space $C(K)$ embeds isometrically into $\ell_\infty /c_0$
\end{fact}
\begin{proof}
 $\sigma_1(\cc)^\omega$ is a continuous image of $\beta\omega\setminus \omega$ (see \cite{BSS}, Theorem 2.5 and Example 5.3)
and since $C(\beta\omega\setminus \omega)$ is
isometric to $\ell_\infty /c_0$, the result follows.
\end{proof}
\begin{ex}\label{example}
A simple counterexample to the question of Y.\ Benyamini, M.E.\
Rudin and M.\ Wage mentioned at the beginning of this section is
the Alexandroff double of the Cantor set $C\subset [0,1]$, which
we will denote by $D$. Since, by a result of J.\ Gerlits from
\cite{G}, the character and the weight coincide for continuous
images of $\sigma_1(\kappa)^\omega$, we conclude that $D$ is not
such an image. It is however a uniform Eberlein compactum. Indeed,
for $x\in C$ and $i=0,1$, let $\Gamma=C\cup\{2\}$ and $f_{x,i}:
\Gamma\to [0,1]$ be defined by the following formula:

$$f_{x,i}(\gamma)=\left\{
    \begin{array}{llll}
        x & \mbox{if } \gamma=2  \\
        0 & \mbox{if } \gamma\in C,\;\gamma\ne x \\
        1 & \mbox{if } \gamma=x,\; i=1 \\
        0 & \mbox{if } \gamma=x,\; i=0
    \end{array}
\right.$$

One can easily verify that the space $\{f_{x,i}: x\in C, \;
i=0,1\}$ considered as a subspace of the product $[0,1]^\Gamma$ is
homeomorphic to the space $D$ (the functions $f_{x,0}$ correspond
to nonisolated points of $D$ and the functions $f_{x,1}$
correspond to the isolated ones). Thus $D$ is a uniform Eberlein
compactum.

The second author was informed about this example by M.\ Bell, who unfortunately had never published it.
As far as we know, in this context, it has never appeared in the literature before.

We should note that the space $D$ was used, in a different
context, by G.~Plebanek in \cite{Plebanek} to distinguish the
class of Eberlein compacta from the class of AD-compacta (the
definition is given below). He noted that $D$ is not a continuous
image of $\sigma_1(\kappa)^\omega$ and is an Eberlein compactum.
He was not aware however of a question of Y.\ Benyamini, M.E.\
Rudin and M.\ Wage.
\end{ex}
\medskip

Another interesting class of compacta is the class of AD-compacta (see \cite{Plebanek}, \cite{bell}).
Given a nonempty set $X$ we say that a family $\mathcal{A}$ of its subsets is \textit{adequate}
if it satisfies the following two conditions:
\begin{itemize}
 \item[(i)] if $A\in \mathcal{A}$ and $B\subseteq A$ then $B\in \mathcal{A}$ and
 \item[(ii)] if $A\subseteq X$ and every finite subset of $A$ is in $\mathcal{A}$, then $A\in \mathcal{A}$.
\end{itemize}
Of course, we can associate in a natural way (identifying a set with its characteristic function) a family $\mathcal{A}$ with a space
$K(\mathcal{A})\subseteq\{0,1\}^X$. It is not difficult to check that $K(\mathcal{A})$ is a compact subspace of $\{0,1\}^X$, provided
$\mathcal{A}$ is adequate.
We say that a compact space $K$ is \textit{adequate} if $K$ is homeomorphic to $K(\mathcal{A})$, for some adequate family $\mathcal{A}$.
We say that a space $K$ is \textit{AD-compact} if it is a continuous image of an adequate compactum.
We refer to \cite{Plebanek}, \cite{bell} for the basic properties of AD-compacta.

M.\ Bell observed in \cite{bell} that consistently there exists an AD-compactum of weight $\leqslant\cc$ which is not a continuous
image of $\beta\omega\setminus\omega$. From our previous considerations we can conclude more
\begin{col}
Suppose that, for every AD-compact space $K$ of weight at most $\cc$, the space $C(K)$ embeds isometrically into $\ell_\infty / c_0$.
Then $\mathfrak{c}$ is a Kunen cardinal.
\end{col}
\begin{proof}
The space $K_2(E_0)$ considered in the proof of \ref{isometr} is adequate compact.
\end{proof}

\begin{col}
 Suppose that, for every AD-compact space $K$ of weight at most $\cc$, the space $C(K)$ embeds isomorphically into $\ell_\infty / c_0$.
Then for every binary relation $E\subseteq \RR$ and for all but finitely many positive integers n,
the set $E^{[n]}$ can be separated from $(E^c)^{[n]}$ by a member of
$\PP^n(\RR)$
\end{col}
\begin{proof}
Since for any set $E\subseteq\RR$ the space $K_n(E)$ is adequate, it is AD-compact and by Theorem 2.1 in \cite{Plebanek}
the space
$K=\bigoplus_{n=1}^{\infty}K_n(E_0)\cup\{\infty\}$ considered in the proof of \ref{isomorph} is also
AD-compact.
\end{proof}


\begin{thebibliography}{99}

\normalsize
\baselineskip=17pt


\bibitem{ARP}
A.\ Avil\'es, G.\ Plebanek, J.\ Rodr\'iguez, {\em Measurability in $C(2^\kappa)$ and Kunen cardinals}, to appear in Israel J. Math.

\bibitem{bell}
M.\ Bell, {\em Generalized dyadic spaces}, Fund. Math. 125 (1985),
47-58

\bibitem{Bell'96}
M.\ Bell, {\em A Ramsey theorem for polyadic spaces}, Fund. Math.
150 (1996), 189-195

\bibitem{BSS}
M.\ Bell, L.\ Shapiro, P.\ Simon, {\em Products of $\omega^\ast$
images}, Proc. Amer. Math. Soc. 124 (1996), 1593-1599

\bibitem{BRW}
Y.\ Benyamini, M.E.\ Rudin, M.\ Wage, {\em Continuous images of
weakly compact subsets of Banach spaces}, Pacific J. Math. 70
(1977), 309-324

\bibitem{BK} C.\ Brech, P.\ Koszmider, \emph{On universal spaces for the class of Banach
spaces whose dual balls are uniform Eberlein compacts}, to appear in Proc. Amer. Math. Soc.

\bibitem{G}
J.\ Gerlits, {\em On a problem of S.\ Mr\' owka}, Period. Math.
Hungar. 4 (1973), 71-80

\bibitem{M}
W.\ Marciszewski, {\em On Banach spaces $C(K)$ isomorphic to
$c_0(\Gamma)$}, Studia Math. 156 (2003), 295-302.

\bibitem{Ne}
S.\ Negropontis, {\em Banach spaces and topology}, in: K.\ Kunen and J.E.\ Vaughan, eds. Handbook of Set-Theoretic Topology
(North-Holland, Amsterdam, 1984) Ch. 23.

\bibitem{Plebanek}
G.\ Plebanek, {\em Compact spaces that result from adequate
families of sets}, Topology Appl. 65 (1995), 257-270

\bibitem{T}
S.\ Todor\v{c}evi\'c, {\em Embedding function spaces into $\ell_\infty / c_0$}, J. Math. Anal. Appl. 384 (2011), 246-251.
\end{thebibliography}
\end{document}